\documentclass[12pt,reqno]{amsart}

\usepackage{amssymb,latexsym}

\usepackage{enumerate}

\usepackage[french,english]{babel}
\usepackage{amsmath}
\usepackage{graphicx}
\usepackage{amssymb}
\usepackage{bbm}
\usepackage{amsthm,mathtools}
\usepackage{ulem}
\usepackage{geometry}
\usepackage{tikz-cd}
\usepackage{mathrsfs}
\usepackage[colorinlistoftodos]{todonotes}
\usepackage{enumitem}
\usepackage{verbatim}
\usepackage[foot]{amsaddr}
\usepackage{dsfont}

\makeatletter

\@namedef{subjclassname@2010}{
	
	\textup{2020} Mathematics Subject Classification}

\author{Zikang Dong\textsuperscript{1}}
\author{Zhonghua Li\textsuperscript{2}}
\author{Yutong Song\textsuperscript{2}}
\author{Shengbo Zhao\textsuperscript{2}}
\address{1. School of Mathematical Sciences, Soochow University, Suzhou 215006, P. R. China}
	\email{zikangdong@gmail.com}
    \email{zhonghua\_li@tongji.edu.cn}
	\email{99yutongsong@gmail.com}
	\email{shengbozhao@hotmail.com}
	\date{\today}
\address{2. School of Mathematical Sciences, Key Laboratory of Intelligent Computing and Applications(Ministry of Education), Tongji University, Shanghai 200092, P. R. China}
\makeatother
\newtheorem{thm}{Theorem}[section]
\newtheorem*{thm*}{Theorem}

\newtheorem{prop}{Proposition}[section]
\newtheorem{lem}[thm]{Lemma}

\theoremstyle{definition}

\newtheorem{rem}[thm]{Remark}

\numberwithin{equation}{section}

\usepackage{hyperref}
\hypersetup{hypertex=true,colorlinks=true,linkcolor=blue,anchorcolor=blue,citecolor=blue}
\newcommand{\newabstract}[1]{%
	\par\bigskip
	\csname otherlanguage*\endcsname{#1}%
	\csname captions#1\endcsname
	\item[\hskip\labelsep\scshape\abstractname.]
}

\begin{document}
\title[Large values of quadratic Dirichlet $L$-functions near the central point]{Large values of quadratic Dirichlet $L$-functions near the central point}
\cleardoublepage \pagestyle{myheadings}
\maketitle

\begin{abstract}
    In this paper, we investigate the conditional large values of the quadratic Dirichlet $L$-functions near the central point $s=1/2$. When $\sigma $ closes to $1/2$ within a suitable range, we show that $L(\sigma, \chi_d)$ have the conditional large values of the similar order of magnitude as $L(1/2, \chi_d)$.
\end{abstract}



\bigskip

\section{Introduction}
The study of the distribution of values of the Riemann zeta function and Dirichlet $L$-functions has been one of the important topics in analytic number theory. 
In recent years, the research on their extreme values has gotten many breakthroughs. Soundararajan \cite{soundararajan2000nonvanishing} proposed a more effective approach, named the resonance method, to replace the previous moment method and derived larger magnitudes of extreme values for the Riemann zeta function as follows 
$$
\max_{T \leq t \leq 2T}\Big| \zeta \Big( \frac{1}{2} + it \Big) \Big| \geq \exp \bigg( (1 + o(1)) \sqrt{\frac{\log T}{\log \log T}} \bigg),
$$
where $T$ is large.
\par
Following Aistleitner's ideas in \cite{aistleitner2016lower}, Bondarenko and Seip \cite{bondarenko2017large,bondarenko2018argument} improved the resonance method and refined Soundararajan's result by a factor of $\sqrt{\log \log \log T}$ in the exponent. Specifically, they obtained the following extreme values
$$
\max_{0 < t \leq T}\Big| \zeta \Big( \frac{1}{2} + it \Big) \Big| \geq \exp \bigg( (1 + o(1)) \sqrt{\frac{\log T \log \log \log T}{\log \log T}} \bigg).
$$
Incidentally, de la Bretèche and Tenenbaum \cite{tenen2019galsum} improved the factor from $(1+o(1))$ to $(\sqrt{2} + o(1))$ in the above result.
\par
The resonance method has an obvious limitation when dealing with similar problems concerning various families of $L$-functions, namely that the coefficients must be positive. Let $X$ be sufficiently large, Soundararajan \cite{soundararajan2008extreme} considered the quadratic Dirichlet $L$-functions and showed
\begin{equation}
    \max _{\substack{X<|d| \leq 2X \\ d \in \mathcal{F}}} L \Big( \frac{1}{2}, \chi_d \Big) \geq \exp \bigg( \left( \frac{1}{\sqrt{5}} + o(1) \right) \sqrt{\frac{{\log X}}{\log \log X}} \bigg). \nonumber
\end{equation}
Here, $\mathcal{F}$ denotes the set of all fundamental discriminants , i.e.
$$
\mathcal{F} := \{d: d \,\, \text{is a fundamental discriminant} \},
$$
and $\chi_d = \big( \frac{d}{\cdot}\big)$ denotes the real primitive character associated to the fundamental discriminant $d$.
\par
The study of the quadratic Dirichlet $L$-functions at the central point can be traced back to the unpublished results of Heath-Brown. He showed that for arbitrarily large $d \in \mathcal{F}$
$$
L \Big( \frac{1}{2}, \chi_d \Big) \gg \exp \bigg( c\frac{\sqrt{\log |d|}}{\log \log |d|} \bigg),
$$
where $c$ is a positive constant. For more details, we refer to \cite{hoffstein66omega}.
\par
Recently, Darbar and Maiti \cite{darbar2025large} employed the long resonator method from \cite{aistleitner2016lower} to obtain conditional large values of the quadratic Dirichlet $L$-functions. They proved that under the Generalized Riemann hypothesis (GRH), one has
$$
\max _{\substack{X<|d| \leq 2X \\ d \in \mathcal{F}}} L \Big( \frac{1}{2}, \chi_d \Big) \geq \exp \bigg( \bigg( \frac{1}{2} + o(1) \bigg) \sqrt{\frac{\log X \log \log \log X}{\log \log X}} \bigg).
$$
Their result matches the form of the currently optimal extreme values of the Riemann zeta function established by Bondarenko and Seip in \cite{bondarenko2017large,bondarenko2018argument}. Under GRH, they improved Soundararajan's result above. Furthermore, for $\sigma \in (1/2, 1)$, they showed that
$$
\max _{\substack{|d| \leq X \\ d \in \mathcal{F}}} L(\sigma, \chi_d) \geq \exp \bigg( (c +o(1)) \frac{(\log X)^{1-\sigma} }{(\log \log X)^{\sigma}}  \bigg),
$$
where $c$ is an explicit constant (See \cite[Lemma 3]{darbar2024large}).
They also computed the proportion of fundamental discriminants $d$ for which large values of $L(\sigma, \chi_d)$ occur when $\sigma \in [1/2, 1]$.
\par
Indeed, based on the random matrix theory, Farmer, Gonek, and Hughes \cite{farmer2007maximum} conjectured that 
$$
\max _{\substack{|d| \leq X \\ d \in \mathcal{F}}} \bigg| L \bigg( \frac{1}{2}, \chi_d \bigg) \bigg| = \exp \left( (1 + o(1)) \sqrt{\log X \log \log X} \right).
$$
For further results and details about the quadratic Dirichlet $L$-functions we refer to \cite{conrey2002real,darbar2024large,granville2003distribution,hoffstein66omega,soundararajan2000nonvanishing} and the references therein.
\par
Inspired by Chirre \cite{chirre2019extreme}, we investigate the conditional large values of the quadratic Dirichlet $L$-functions near the central point. Following what was done in \cite{dong2025Onde}, we investigate the conditional large values of $L(\sigma, \chi_d)$ when $0 < \sigma - 1/2 \ll (\log \log X)^{-1}$. We utilize the long resonator method and derive the following Theorem \ref{thm1.1}.
\begin{thm}
    \label{thm1.1}
    Assume GRH is true. Let $A$ be any positive real number and $X$ be sufficiently large. Set
    $$\sigma=\frac{1}{2} + \frac{A}{\log \log X},$$
    then we have
    \begin{equation}
        \label{thm1}
        \max _{\substack{X<|d| \leq 2X \\ d \in \mathcal{F}}} L(\sigma, \chi_d) \geq \exp \bigg((\lambda(A)+o(1)) \sqrt{\frac{\log X  \log \log \log X }{\log \log X }}\bigg),
    \end{equation}
    where 
    $$\lambda(A) = \frac{1}{2(e-1)e^A}.$$
\end{thm}
\par
Compared with \cite[Theorem 1]{darbar2025large}, our Theorem \ref{thm1.1} shows that when $0< \sigma - 1/2 \ll (\log \log X)^{-1}$, the conditional large values of $L(\sigma, \chi_d)$ can reach the similar order of magnitude as those of $L(1/2, \chi_d)$ as $X \to \infty$, though the leading constant in the exponent is slightly smaller. At the end of the proof, we will explain the reason for the emergence of the smaller constant, and also clarify the possibility of improving it.
\par
To simplify the writing, we adopt the following short-hand notations in this paper, such as $\log_2 x \operatorname{:} = \log{\log{x}}$ and $ \log_3 x \operatorname{:} = \log{\log{\log x}}.$ Moreover, let $\mathbb{N}$ be the set of all positive integers. And let $\varepsilon > 0$ be an arbitrarily small number. Note that each appearance of $\varepsilon$ does not necessarily denote the same value. Finally, we use $\square$ to indicate a perfect square number.

\section{Preliminary Lemmas}
In this section, we give some lemmas that we will use later. The first lemma below serves to handle the partial sums of $\chi_d(n)$.
\begin{lem}[\cite{darbar2025large}, Lemma 1]
    \label{lemmadarbar}
    Assume that GRH is true. Let $n=n_0 n_1^2$ be a positive integer with $n_0$ the square-free part of $n$. Then for any $\varepsilon>0$, we obtain
    $$
    \sum_{\substack{|d| \leq X \\ d \in \mathcal{F}}} \chi_d(n)=\frac{X}{\zeta(2)} \prod_{p \mid n}\Big(\frac{p}{p+1}\Big) \mathrm{1}_{n=\square}+O\Big(X^{\frac{1}{2}+\varepsilon} g_1\left(n_0\right) g_2\left(n_1\right)\Big),
    $$
    where $g_1(n_0)=\exp \big(\left(\log n_0\right)^{1-\varepsilon} \big)$, $g_2\left(n_1\right)=\sum_{q \mid n_1} \mu^2(q) q^{-(1 / 2+\varepsilon)}$, and $\mathrm{1}_{n=\square}$ indicates the indicator function of the square numbers.
\end{lem}
\par
The next one provides the approximate functional equation of the quadratic Dirichlet $L$-functions within the critical strip.
\begin{lem}[\cite{young2009}, Proposition 2.1]
     \label{lemmayoung}
      Let $\alpha$ be a positive parameter. And let $G(s)$ be any entire, even function bounded in any strip $|\operatorname{Re}(s)| \leq M$ which normalized by $G(0)=1$. Then for $d \in \mathcal{F}$, we have
      \begin{equation}
           L\big(\tfrac{1}{2}+\alpha, \chi_d \big)=\sum_{n \geq 1} \frac{\chi_{d}(n)}{n^{\frac{1}{2}+\alpha}} U_\alpha \Big(\frac{n}{\sqrt{d}} \Big)+Y_\alpha \sum_{n \geq 1} \frac{\chi_{d}(n)}{n^{\frac{1}{2}-\alpha}} U_{-\alpha}\Big(\frac{n}{\sqrt{d}}\Big),
      \end{equation}
      where 
      $$U_\alpha(x)=\frac{1}{2 \pi i} \int_{(1)} \frac{G(s)}{\pi^{\frac{s}{2}}} \frac{\Gamma\big(\frac{\frac{1}{2}+\alpha+s}{2}\big)}{\Gamma \big(\frac{\frac{1}{2}+\alpha}{2} \big)} x^{-s} \frac{\mathrm{d}s}{s} \quad \text{and} \quad Y_\alpha=\Big(\frac{d}{\pi}\Big)^{-\alpha} \frac{\Gamma\big(\frac{\frac{1}{2}-\alpha}{2}\big)}{\Gamma\big(\frac{\frac{1}{2}+\alpha}{2}\big)}.$$
    
\end{lem}

\section{Proof of Theorem \ref{thm1.1}}
\subsection{Constructing the resonator}
We use the resonance method to study large values of $L$-functions near the central point. Let $N$ be a large number that will be chosen later, and 
$$\sigma = \frac{1}{2} + \frac{A}{\log_ 2 N}.$$
Following the ideas of \cite{bondarenko2017large,bondarenko2018argument,bondarenko2018note,chirre2019extreme,darbar2025large}, we proceed with the following construction.
 Let $\gamma \in (0,(e-1)^{-1})$ be a parameter to be chosen later. Define $\mathcal{P}$ as the set of prime numbers $p$ such that
$$
e \log N \log_{2}N < p \leq \exp\left(\left(\log_2 N\right)^{\gamma}\right) \log N \log_{2}N.
$$
Then define the multiplicative function $f(n)$ to be supported on the set of square-free numbers, with values for $p \in \mathcal{P}$ as
$$
f(p) = \bigg( \frac{(\log N)^{1 - \sigma} (\log_2 N)^{\sigma}}{(\log_3 N)^{1 - \sigma}} \bigg) \frac{1}{p^{\sigma} (\log p - \log_2 N - \log_3 N )}.
$$
If $p \notin \mathcal{P}, \,f(p)=0.$ 
For  $k \in \left\{1, \cdots, \lfloor(\log_2 N)^{\gamma}\rfloor\right\},$ we define $\mathcal{P}_{k}$ as the set of prime numbers $p$ such that 
$$e^k \log N \log_{2}N < p \leq e^{k+1} \log N\log_{2}N.$$
Fix a real number $a$ with the condition $e-1 < a < 1/\gamma,$ then define 
$$\mathcal{M}_{k} := \left\{ n \in \mathrm{supp}(f) : n \text{ has at least } \Delta_k := \frac{a (\log N)^{2 - 2\sigma}}{k^2 (\log_3 N)^{2 - 2\sigma}} \text{ prime divisors in } \mathcal{P}_{k} \right\}.$$ 
Let $\mathcal{M}^{\prime}_{k}$ be the set of integers from $\mathcal{M}_{k}$ that have prime divisors only in $\mathcal{P}_{k}.$ Set
$$
\mathcal{M} := \mathrm{supp}(f) \setminus \bigcup_{k = 1}^{\lfloor(\log_2 N)^{\gamma}\rfloor} \mathcal{M}_{k}.
$$
Thus, $\mathcal{M}$ is divisor closed. Specifically, if $m^{\prime} \mid m \in \mathcal{M},$ then $m^{\prime} \in \mathcal{M}.$ Furthermore, according to the similar process as \cite[Lemma 8]{chirre2019extreme}, we have 
\begin{align*}
    |\mathcal{M}| & \leq \exp ((a \gamma + o(1))(\log N)^{2-2\sigma}(\log_3 N)^{2\sigma -1} ) \\
    & \leq \exp ((a \gamma + o(1))\log N(\log_3 N)^{\frac{2A}{\log _2 N}}) \leq N.
\end{align*}
Finally, in view of the definition of $\mathcal{M},$ we can bound every $m \in \mathcal{M}$ by 
\begin{equation}
    \label{factorMbound}
    m \ll \exp \bigg(\frac{\log N \log_3 N}{\log_2 N} \bigg).
\end{equation}
\par
According to the heuristic description by Bondarenko and Seip in \cite{bondarenko2017large}, it is further necessary to construct a set $\mathcal{M}^{\prime}$ such that certain parts contribute larger. However, by \cite{darbar2025large}, our choice is simpler compared with the constructions in \cite{bondarenko2017large,chirre2019extreme}. Directly setting $\mathcal{M}^{\prime}$ to be identical with the above $\mathcal{M}$ enables us to prove Theorem \ref{thm1.1}. 
\par
We define the resonator $R_d$ as 
$$
R_d := \sum_{m \in \mathcal{M}} r(m) \chi_d(m).
$$
Then for $n \in \mathcal{M}^{\prime}$ if we set 
$$
r(n)^2=\sum_{\substack{m \in \mathcal{M}^{\prime} \\ mn=\square}} f(m)^2,$$ by the construction of $\mathcal{M}$ we have $r(n)=f(n),$ so 
$$
R_d = \sum_{m \in \mathcal{M}} f(m) \chi_d(m).
$$
\par
We define multiplicative function $h(n)=\prod_{p \mid n} \big(\frac{p}{p+1}\big)$ for $n \geq 2$ and $h(1)=1.$ 
Set
\begin{align}
    \mathcal{A}_{N} &:= \frac{1}{\sum_{n\in\mathbb{N}} f(n)^2} \sum_{n\in\mathbb{N}} \frac{f(n)}{n^{\sigma}}h(n) \sum_{q\mid n} f(q)q^{\sigma} \nonumber \\
    & = \prod_{p\in \mathcal{P}} \frac{1 + f(p)^2 h(p) + f(p)h(p)p^{-\sigma}}{1 + f(p)^2}.
\end{align}
Then we have the following propositions.
\begin{prop}
    \label{prop3.1}
    We have
    $$\mathcal{A}_{N} \geq \exp\bigg((\delta\gamma+ o(1))\frac{\left(\log N  \right)^{1-\sigma} \left(\log_3 N\right)^{\sigma}}{\left(\log_2 N\right)^{\sigma}}\bigg)$$
for $0 < \delta < 1$ as $N \to +\infty.$
\end{prop}
\begin{proof}
    Since $f(p) < \left( \log_3 N\right)^{\sigma-1}$ for $p \in \mathcal{P},$ we obtain
    $$
    \mathcal{A}_{N} =\exp \bigg((1+o(1)) \sum_{p \in \mathcal{P}} \frac{f(p)}{p^{\sigma}}\Big(\frac{p}{p+1}\Big)\bigg)=\exp \bigg((1+o(1)) \sum_{p \in \mathcal{P}} \frac{f(p)}{p^{\sigma}}\bigg) .
    $$
    By the definition of $f(p)$, we have
      \begin{align*}
      \sum_{p \in \mathcal{P}} \frac{f(p)}{p^{\sigma}} 
      & =\frac{(\log N)^{1 - \sigma} (\log_2  N)^{\sigma}}{(\log_3 N)^{1 - \sigma}} \sum_{p \in \mathcal{P}} \frac{1}{p^{2\sigma}\left(\log p-\log _2 N-\log _3 N\right)}  \\
      & \geq \frac{(\log N)^{1 - \sigma} (\log_2 N)^{\sigma}}{(\log_3 N)^{1 - \sigma}} \sum_{k=1}^{\lfloor(\log_2 N)^{\gamma}\rfloor} \sum_{p \in \mathcal{P}_k} \frac{1}{(k+1)p^{2\sigma}}.
      \end{align*}
    By \cite[Lemma 9]{chirre2019extreme}, we get the following lower bound
    $$
    \sum_{p \in \mathcal{P}_k} \frac{1}{p^{2\sigma}} > (\delta + o(1))\frac{1}{(\log_2 N)^{ 2\sigma}},
    $$
    where $0< \delta <1.$ So
    $$
    \sum_{p \in \mathcal{P}} \frac{f(p)}{p^{\sigma}} \geq (\delta\gamma+ o(1))\frac{\left(\log N  \right)^{1-\sigma} \left(\log_3 N\right)^{\sigma}}{\left(\log_2 N\right)^{\sigma}}
    $$
    and we complete the proof.
\end{proof}

\begin{prop}
    \label{prop3.2}
    We have
    $$\frac{1}{\sum_{n \in \mathbb{N}} f(n)^2} \sum_{\substack{n \in \mathbb{N} \\ n \notin \mathcal{M}}} \frac{f(n) h(n)}{n^{\sigma}} \sum_{q \mid n} f(q)q^{\sigma}=o\left(\mathcal{A}_N\right) $$
    as $N \to +\infty.$
\end{prop}

\begin{proof}
    Since $h(n) \leq 1,$ we obtain
    $$
    \frac{1}{\sum_{n \in \mathbb{N}} f(n)^2} \sum_{\substack{n \in \mathbb{N} \\ n \notin \mathcal{M}}} \frac{f(n) h(n)}{n^{\sigma}} \sum_{q \mid n} f(q)q^{\sigma}
    \leq \frac{1}{\sum_{n \in \mathbb{N}} f(n)^2} \sum_{\substack{n \in \mathbb{N} \\ n \notin \mathcal{M}}} \frac{f(n)}{n^{\sigma}} \sum_{q \mid n} f(q)q^{\sigma}.
    $$
    Next, following the similar way as \cite[Lemma 2]{bondarenko2017large}, we note that
    \begin{align*}
      & \frac{1}{\mathcal{A}_N \sum_{n \in \mathbb{N}} f(n)^2} \sum_{n \in \mathbb{N}, n \notin \mathcal{M}} \frac{f(n)}{n^{\sigma}} \sum_{q \mid n} f(q)q^{\sigma} \\
      & \quad \leq \frac{1}{\mathcal{A}_N \sum_{n \in \mathbb{N}} f(n)^2} \sum_{k=1}^{\lfloor(\log_2 N)^{\gamma}\rfloor} \sum_{n \in \mathcal{M}_k} \frac{f(n)}{n^{\sigma}} \sum_{q \mid n} f(q)q^{\sigma} \\
      & \quad =\sum_{k=1}^{\lfloor(\log_2 N)^{\gamma}\rfloor}\frac{1}{\prod_{p \in \mathcal{P}_k}\left(1+f(p)^2+f(p) p^{-\sigma}\right)} \sum_{n \in \mathcal{M}_k^{\prime}} \frac{f(n)}{n^{\sigma}} \sum_{q \mid n} f(q)q^{\sigma} \\
      & \quad \leq \sum_{k=1}^{\lfloor(\log_2 N)^{\gamma}\rfloor} \frac{1}{\prod_{p \in \mathcal{P}_k}\left(1+f(p)^2\right)} \sum_{n \in \mathcal{M}_k^{\prime}} f(n)^2 \prod_{p \in \mathcal{P}_k}\Big(1+\frac{1}{f(p)p^{\sigma}}\Big).
\end{align*}
By the definition of the set $\mathcal{P}_k$ and $k \leq \lfloor(\log_2 N)^{\gamma}\rfloor,$ we can get
\begin{align}
     \label{o1}
    \prod_{p \in \mathcal{P}_k}\Big(1+\frac{1}{f(p)p^{\sigma}}\Big) & \leq \bigg( 1+(k+1)\frac{(\log_3 N)^{1-\sigma}}{(\log N)^{1-\sigma}(\log_2 N)^{\sigma}}\bigg)^{e^{k+1}\log N} \nonumber \\
    &=\exp \bigg(o\bigg(\frac{(\log N)^{2-2\sigma}}{(\log _3 N)^{2-2\sigma}}\bigg) \frac{1}{k^2}\bigg).
\end{align}

\par
Then by the definition of the set $\mathcal{M}^{\prime}_{k}$ and the fact that $f(n)$ is a multiplicative function, we obtain
$$
\sum_{n\in\mathcal{M}^{\prime}_{k}} f(n)^2 \leqslant b^{-\Delta_k} \prod_{p\in \mathcal{P}_{k}} (1 + b f(p)^2)
$$
whenever $b>1.$ Thus,
$$
\frac{1}{\prod_{p \in \mathcal{P}_k}\left(1+f(p)^2\right)} \sum_{n \in \mathcal{M}_k^{\prime}} f(n)^2 \leq b ^{-\Delta_k} \exp \Big( (b - 1)\sum_{p\in \mathcal{M}^{\prime}_{k}}f(p)^2\Big).
$$
Recalling the definition of $f(p)$ for $p \in \mathcal{P},$ we have 
\begin{align*}
      \sum_{p \in \mathcal{P}_k} f(p)^2 
      & =\frac{(\log N)^{2 - 2\sigma} (\log_2  N)^{2\sigma}}{(\log_3 N)^{2 - 2\sigma}} \sum_{p \in \mathcal{P}_k} \frac{1}{p^{2\sigma}\left(\log p-\log _2 N-\log _3 N\right)^2}  \\
      & \leq \frac{(\log N)^{2 - 2\sigma} (\log_2  N)^{2\sigma}}{k^2 (\log_3 N)^{2 - 2\sigma}} \sum_{p \in \mathcal{P}_k} \frac{1}{p^{2\sigma}}.
      \end{align*}
According to \cite[Lemma 9]{chirre2019extreme}, 
$$
\sum_{p \in \mathcal{P}_k} f(p)^2 \leq \left((e-1)+o(1)\right) \frac{(\log N)^{2 - 2\sigma}}{k^2 (\log_3 N)^{2 - 2\sigma}}.
$$
\par
Combining with \eqref{o1}, 
\begin{align*}
    \frac{1}{\mathcal{A}_N \sum_{n \in \mathbb{N}} f(n)^2} & \sum_{n \in \mathbb{N}, n \notin \mathcal{M}} \frac{f(n)}{n^{\sigma}} \sum_{q \mid n} f(q)q^{\sigma} \\
    & \leq \sum_{k=1}^{\lfloor(\log_2 N)^{\gamma}\rfloor} \exp \Big(\big((e-1)(b-1)-a\log b+o(1)\big) \frac{(\log N)^{2 - 2\sigma}}{k^2 (\log_3 N)^{2 - 2\sigma}} \Big).
\end{align*}
Noticing that $a > e-1,$ we can choose $b$ to be sufficiently close to 1 such that the exponent is negative. The proof is completed.
\end{proof}

\begin{prop}
    \label{prop3.3}
    We have
    $$\frac{1}{\sum_{n \in \mathbb{N}} f(n)^2} \sum_{n \in \mathbb{N}} \frac{f(n) h(n)}{n^{\sigma}} \sum_{\substack{q \mid n \\ q \leq n/N^{\varepsilon} }} f(q)q^{\sigma}=o\left(\mathcal{A}_N\right) 
    $$
    as $N \to +\infty,$ where the implicit constant only depends on $\varepsilon.$
\end{prop}

\begin{proof}
    Since $h(n) \leq 1,$ we obtain
    $$
    \frac{1}{\sum_{n \in \mathbb{N}} f(n)^2} \sum_{n \in \mathbb{N}} \frac{f(n) h(n)}{n^{\sigma}} \sum_{\substack{q \mid n \\ q \leq n/N^{\varepsilon} }} f(q)q^{\sigma} \leq \frac{1}{\sum_{n \in \mathbb{N}} f(n)^2} \sum_{n \in \mathbb{N}} \frac{f(n)}{n^{\sigma}} \sum_{\substack{q \mid n \\ q \leq n/N^{\varepsilon} }} f(q)q^{\sigma}.
    $$
    Given that $f(n)$ is a multiplicative function and $\sigma \geq 1/2,$ the proof can then be completed by invoking the same method as \cite[Lemma 3]{bondarenko2017large}.
\end{proof}

\subsection{Proof of Theorem \ref{thm1.1}}
For convenience, we let $\sigma = 1/2 + \sigma_A$, where $\sigma_A = A/\log_2 N$. 
\par
To use the resonance method, we define the following two sums
$$
\mathcal{S}_1:= \mathcal{S}_1(R_d, \chi_d) = \sum_{\substack{X<|d| \leq 2 X \\ d \in \mathcal{F}}} L\Big(\frac{1}{2}+\sigma_A, \chi_d\Big) R_d^2 
$$
and
$$
\mathcal{S}_2:= \mathcal{S}_2(R_d, \chi_d) = \sum_{\substack{X<|d| \leq 2 X \\ d \in \mathcal{F}}} R_d^2 .
$$
Plainly, we have 
\begin{equation}
    \label{maxL}
    \max_{\substack{X<|d| \leq 2X \\ d \in \mathcal{F}}} L\Big(\frac{1}{2}+\sigma_A, \chi_d\Big) \geq \frac{S_1}{S_2}.
\end{equation}
First, we need to deal with $\mathcal{S}_1$ and obtain its effective lower bound. Employing the definition of $R_d$ and Lemma \ref{lemmayoung}, we get 
\begin{align}
    \label{S1equation}
    \mathcal{S}_1 &=  \sum_{m, n \in \mathcal{M}} f(m) f(n) \sum_{\ell \geq 1} \frac{1}{{\ell}^{\frac{1}{2}+\sigma_A}} \sum_{\substack{X<|d| \leq 2 X \\ d \in \mathcal{F}}} \chi_d(\ell m n) U_{\sigma_A}\Big(\frac{\ell}{\sqrt{d}}\Big) \nonumber \\
    & \quad + Y_{\sigma_A} \sum_{m, n \in \mathcal{M}} f(m) f(n) \sum_{\ell \geq 1} \frac{1}{{\ell}^{\frac{1}{2}-\sigma_A}} \sum_{\substack{X<|d| \leq 2 X \\ d \in \mathcal{F}}} \chi_d(\ell m n) U_{-\sigma_A}\Big(\frac{\ell}{\sqrt{d}}\Big) \nonumber \\
    & := J_{\sigma_A} + Y_{\sigma_A} J_{\sigma_A}.
\end{align}
\par
To handle each term in \eqref{S1equation}, we estimate the integral $U_{\sigma_A} (x)$. Since the function $G(s)$ in Lemma \ref{lemmayoung} is bounded, we may without loss of generality assume that $G(1) = 1$. Using a similar method as \cite[Lemma 2.1]{soundararajan2000nonvanishing}, we can obtain the following proposition.
\begin{prop}
    \label{Uproperties}
    Let $U_\alpha(x)$ be the integral defined in Lemma \ref{lemmayoung}. Then $U_\alpha(x)$ is a real-valued function, and is smooth on the positive real axis $(0, + \infty).$ Moreover, for small $x, \,U_\alpha(x)= 1 + O(x^{1/2-\varepsilon}),$ and for large $x, \,\,U_\alpha(x) \ll e^{-x}.$
\end{prop}
\par
Using Lemma \ref{lemmadarbar}, we divide the inner sum of $J_{\sigma_A}$ into two cases: $\ell m n =\square$ and $\ell m n \neq \square.$ Then by the upper bound of $U_{\sigma_A}(x)$ in Proposition \ref{Uproperties} we obtain
\begin{equation}
    J_{\sigma_A}=  \frac{X}{\zeta(2)} \sum_{m, n \in \mathcal{M}} f(m)f(n) \sum_{\substack{\ell \geq 1 \\
\ell m n=\square}} \frac{1}{\ell^{\frac{1}{2}+\sigma_A}} \prod_{p \mid \ell m n}\Big(\frac{p}{p+1}\Big) \int_1^2 U_{\sigma_A}\Big(\frac{\ell}{\sqrt{X t}}\Big) \mathrm{d}t +O\left(\mathcal{E}\right),
\end{equation}
where the error term is
$$
\mathcal{E}=X^{\frac{1}{2}+\varepsilon} \sum_{\ell \geq 1} \frac{e^{-\ell X^{\frac{1}{2}}}}{\ell^{\frac{1}{2}+\sigma_A}} \sum_{\substack{m, n \in \mathcal{M} \\
\ell m n=k_1 k_2^2 \\
\mu\left(k_1\right) \neq 0}} f(m) f(n) g_1\left(k_1\right) g_2\left(k_2\right).
$$
Since \eqref{factorMbound}, we have the following upper bound 
\begin{equation}
    \label{g1upperbound}
    g_1(k_1) =\exp \big(\left(\log k_1\right)^{1-\varepsilon}\big) \ll \exp \Big(\Big(\frac{\log N \log _2 N}{\log _3 N}\Big)^{1-\varepsilon}\Big) \ll N^{\frac{\varepsilon}{2}}.
\end{equation}
Furthermore, it follows from the definition of the Möbius function that $g_2$ is a multiplicative function and satisfies $g_2(xy)\leq g_2(x)g_2(y)$ for $x,y\geq 1.$ Thus, $\ell m n=k_1 k_2^2$ yields that $g_2(k_2)\leq g_2(\ell)g_2(m)g_2(n).$ By \cite[Eq. (12)]{darbar2025large}, for all $m \in \mathcal{M},$ we get
\begin{equation}
    \label{g2upperbound}
    g_2(m)=\sum_{q \mid m} \frac{\mu^2(q)}{q^{1 / 2+\varepsilon}} \ll \exp\Big(\sum_{p \leq \exp(\log_2 N)^{\gamma}\log N \log_2 N} p^{-(\frac{1}{2}+\varepsilon)} \Big) \ll N^{\frac{\varepsilon}{2}}.
\end{equation}
\par
Combining \eqref{g1upperbound}, \eqref{g2upperbound} and using the Cauchy-Schwarz inequality, the error term 
\begin{equation}
    \label{mathcalEupper}
    \mathcal{E} \ll X^{\frac{1}{2}+\varepsilon} N^{\frac{3\varepsilon}{2}}|\mathcal{M}|\Big(\sum_{\ell \geq 1} \frac{e^{-\ell X^{-\frac{1}{2}}} }{\ell^{\frac{1}{2}+\sigma_A}}g(\ell)\Big)\Big(\sum_{n \in \mathcal{M}} f(n)^2\Big).
\end{equation}
Applying the trivial bound of $g_2(\ell),$ we have
$$
\sum_{\ell \geq 1} \frac{e^{-\ell X^{-\frac{1}{2}}}}{\ell^{\frac{1}{2}+\sigma_A}}g(\ell) \ll X^{\varepsilon} \sum_{\ell \leq X^{\frac{1}{2}+\varepsilon}} \ell^{-\frac{1}{2}} \ll X^{\frac{1}{4}+\varepsilon} .
$$
Returning to \eqref{mathcalEupper}, we obtain an effective upper bound for the error term as follows
\begin{equation}
    \mathcal{E} \ll X^{\frac{3}{4}+2\varepsilon} N^{\frac{3\varepsilon}{2}}|\mathcal{M}|\Big(\sum_{n \in \mathcal{M}} f(n)^2\Big).
\end{equation}
\par
Next, we consider the main term of $J_{\sigma_A}.$ Since both functions $f(n)$ and $U_{\sigma_A}(x)$ are always positive, we crudely pick out all terms of the form $\ell m \neq n$ and get
$$
J_{\sigma_A} \geq \frac{X}{\zeta(2)} \sum_{n \in \mathcal{M}} \frac{f(n)}{n^{\frac{1}{2}+\sigma_A}} \prod_{p \mid n}\Big(\frac{p}{p+1}\Big) \sum_{m \mid n} f(m) m^{\frac{1}{2}+\sigma_A} \int_1^2 U_{\sigma_A}\Big(\frac{n}{m \sqrt{X t}}\Big) d t  +O\left(\mathcal{E}\right).
$$
Utilizing Proposition \ref{Uproperties} gives 
\begin{equation}
    \label{Jalpha}
    J_{\sigma_A} \geq \frac{X}{\zeta(2)} \sum_{n \in \mathcal{M}} \frac{f(n) h(n)}{n^{\frac{1}{2}+\sigma_A}} \sum_{\substack{m \mid n \\ m \geq n/N^{\varepsilon}}} f(m)m^{\frac{1}{2}+\sigma_A} + O\Big(X^{\frac{3}{4}+2\varepsilon} N^{1+ \frac{3\varepsilon}{2}} \Big(\sum_{n \in \mathcal{M}} f(n)^2\Big)\Big).
\end{equation}
\par
Using arguments similar to those above, we can estimate the lower bound of $J_{-\sigma_A}$ and get
\begin{align}
    \label{Jfualphalower}
    J_{-\sigma_A} \geq  \frac{X}{\zeta(2)} \sum_{n \in \mathcal{M}} \frac{f(n) h(n)}{n^{\frac{1}{2}-\sigma_A}} \sum_{\substack{m \mid n \\ m \geq n/N^{\varepsilon}}} f(m)m^{\frac{1}{2}-\sigma_A}
    + O\Big(X^{\frac{3}{4}+\frac{\sigma_A}{2}+2\varepsilon} N^{1+ \frac{3\varepsilon}{2}} \Big(\sum_{n \in \mathcal{M}} f(n)^2\Big)\Big).
\end{align}

For $n \in \mathcal{M}$, we have $n^{2\sigma_A} \geq 1.$ And for $m \mid n,$ we have $m \in \mathcal{M}.$ Furthermore, \eqref{factorMbound} yields that
$$
m^{-2\sigma_A} \gg \exp \Big(-2A\frac{\log N}{\log_3 N} \Big).
$$
Combining with \eqref{Jfualphalower}, we obtain
\begin{align}
    \label{Jfualpha}
    J_{-\alpha} & \gg \frac{X}{\zeta(2)} \exp \Big(-2A\frac{\log N}{\log_3 N} \Big) \sum_{n \in \mathcal{M}} \frac{f(n) h(n)}{n^{\frac{1}{2}+\sigma_A}} \sum_{\substack{m \mid n \\ m \geq n/N^{\varepsilon}}} f(m)m^{\frac{1}{2}+\sigma_A} \nonumber \\
    & + O\Big(X^{\frac{3}{4}+\frac{\sigma_A}{2}+2\varepsilon} N^{1+ \frac{3\varepsilon}{2}} \Big(\sum_{n \in \mathcal{M}} f(n)^2\Big)\Big).
\end{align}
\par
Thus \eqref{S1equation}, \eqref{Jalpha} and \eqref{Jfualpha} give
\begin{align*}
    \mathcal{S}_1 & \gg \Big( 1+ Y_{\sigma_A} \exp \Big(-2A\frac{\log N}{\log_3 N} \Big) \Big) \frac{X}{\zeta(2)}  \sum_{n \in \mathcal{M}} \frac{f(n) h(n)}{n^{\frac{1}{2}+\sigma_A}} \sum_{\substack{m \mid n \\ m \geq n/N^{\varepsilon}}} f(m)m^{\frac{1}{2}+\sigma_A} \\
    & + O\Big(X^{\frac{3}{4}+\frac{\sigma_A}{2}+2\varepsilon} N^{1+ \frac{3\varepsilon}{2}} \Big(\sum_{n \in \mathcal{M}} f(n)^2\Big)\Big).
\end{align*}
Through simple calculations, it can be seen that $ Y_{\sigma_A}$ is bounded. Similarly to \cite{darbar2025large}, choosing $N \asymp X^{\left(1/4 + \sigma_A /2 -5 \varepsilon\right)}$, where $0 < \varepsilon < 1/20$ is arbitrary small, and using Proposition \ref{prop3.1}, Proposition \ref{prop3.2} and Proposition \ref{prop3.3}, we have
\begin{align}
     \label{S1lowerbound}
     \mathcal{S}_1 &\geq (1+o(1)) \frac{X}{\zeta(2)} \nonumber \\
     & \cdot \exp \bigg( \Big(\delta \gamma \Big(\frac{1}{4} + \frac{\sigma_A}{2}- 5 \varepsilon \Big)^{1-\sigma}+o(1)\Big) \frac{\left(\log N  \right)^{1-\sigma} \left(\log_3 N\right)^{\sigma}}{\left(\log_2 N\right)^{\sigma}}\bigg) \Big(\sum_{n \in \mathcal{M}} f(n)^2\Big).
\end{align}
\par
Following the same process as in \cite{darbar2025large}, we can obtain the upper bound of $\mathcal{S}_2$ as follows
\begin{equation}
   \label{S2upperbound}
   \mathcal{S}_2 \leq(1+o(1)) \frac{X}
 {\zeta(2)} \Big(\sum_{n \in \mathcal{M}} f(n)^2\Big).
\end{equation}
Thus, combining \eqref{maxL}, \eqref{S1lowerbound} and \eqref{S2upperbound}, we get
\begin{align*}
     \max_{\substack{X<|d| \leq 2X \\ d \in \mathcal{F}}} L\left(\sigma, \chi_d\right)  & \geq  \exp \bigg( \Big(\delta \gamma \Big(\frac{1}{4} + \frac{\sigma_A}{2} - 5 \varepsilon \Big)^{1-\sigma}+o(1)\Big) \frac{\left(\log N  \right)^{1-\sigma} \left(\log_3 N\right)^{\sigma}}{\left(\log_2 N\right)^{\sigma}}\bigg). \\
     & \geq  \exp \bigg( \Big(\delta \gamma \Big(\frac{1}{4} + \frac{\sigma_A}{2} - 5 \varepsilon \Big)^{1-\sigma} e^{-A} +o(1)\Big) \sqrt{\frac{\log X \log_3 X}{\log_2 X}}\bigg).
\end{align*}

Choosing
$\delta = 1 + o(1), \gamma = (e-1)^{-1} + o(1)$
and noticing that for sufficiently large $N,$
$$
\Big(\frac{1}{4} + \frac{\sigma_A}{2} - 5 \varepsilon \Big)^{1-\sigma} = \frac{1}{2} + o(1),
$$
we complete the proof of Theorem \ref{thm1.1}.

\begin{rem}
    Our Lemma \ref{lemmayoung} is derived from \cite[Proposition 2.1]{young2009}, which is a simplified version. A stronger one can be found in \cite[Theorem 5.3]{iwaniec2004analytic}.
\end{rem}
\begin{rem}
    In the proof of Proposition \ref{prop3.2}, we employed a somewhat crude estimate \cite[Lemma 9]{chirre2019extreme}, which led to a slightly strange constant $(e-1)^{-1}.$ If a more refined estimate of $\sum p^{-2\sigma}$ could be made, the coefficient in the exponent in Theorem \ref{thm1.1} would be slightly improved.
\end{rem}
\begin{rem}
    In the proof of Proposition \ref{Uproperties}, we assumed that the function $G(s)$ in Lemma \ref{lemmayoung} satisfies $G(1) = 1.$ This is because $G(s)$ is bounded, so regardless of its value, it will be absorbed by the term $o(1)$ in the exponent in \eqref{thm1}.
\end{rem}

\bibliographystyle{siam}
\bibliography{reference}

\begin{thebibliography}{10}

\bibitem{aistleitner2016lower}
{\sc C.~Aistleitner}, {\em Lower bounds for the maximum of the {R}iemann zeta function along vertical lines}, Math. Ann., 365 (2016), pp.~473--496.

\bibitem{bondarenko2017large}
{\sc A.~Bondarenko and K.~Seip}, {\em Large greatest common divisor sums and extreme values of the {R}iemann zeta function}, Duke Math. J., 166 (2017), pp.~1685--1701.

\bibitem{bondarenko2018argument}
{\sc A.~Bondarenko and K.~Seip}, {\em Extreme values of the {R}iemann zeta function and its argument}, Math. Ann., 372 (2018), pp.~999--1015.

\bibitem{bondarenko2018note}
{\sc A.~Bondarenko and K.~Seip}, {\em Note on the resonance method for the {R}iemann zeta function}, Oper. Theory Adv.Appl., 261 (2018), pp.~121--139.

\bibitem{chirre2019extreme}
{\sc A.~Chirre}, {\em Extreme values for ${S}_n (\sigma, t)$ near the critical line}, J. Number Theory, 200 (2019), pp.~329--352.

\bibitem{conrey2002real}
{\sc J.~B. Conrey and K.~Soundararajan}, {\em Real zeros of quadratic {D}irichlet ${L}$-functions}, Invent. Mtah., 150 (2002), pp.~1--44.

\bibitem{darbar2024large}
{\sc P.~Darbar and G.~Maiti}, {\em Large values of quadratic {D}irichlet ${L}$-functions over monic irreducible polynomial in $\mathbb{F}_q [t]$}, Proc. Am. Math. Soc., 152 (2024), pp.~3243--3254.

\bibitem{darbar2025large}
{\sc P.~Darbar and G.~Maiti}, {\em Large values of quadratic {D}irichlet ${L}$-functions}, Math. Ann.,  (2025), pp.~1--33.

\bibitem{tenen2019galsum}
{\sc R.~de~la Bret{\`e}che and G.~Tenenbaum}, {\em Sommes de {G}\'al et applications}, Proc. London Math. Soc., 119 (2019), pp.~104--134.

\bibitem{dong2025Onde}
{\sc Z.~Dong, Y.~Song, W.~Wang, and H.~Zhang}, {\em On derivatives of zeta and ${L}$-functions}, Ramanujan J., 66 (2025), pp.~5--21.

\bibitem{farmer2007maximum}
{\sc D.~W. Farmer, S.~M. Gonek, and C.~P. Hughes}, {\em The maximum size of ${L}$-functions}, J. Reine Angew. Math.,  (2007), pp.~215--236.

\bibitem{granville2003distribution}
{\sc A.~Granville and K.~Soundararajan}, {\em The distribution of values of ${L}(1, \chi_d)$}, Geom. Funct. Anal., 13 (2003), pp.~992--1028.

\bibitem{hoffstein66omega}
{\sc J.~Hoffstein and P.~Lockhart}, {\em Omega results for automorphic ${L}$-functions. {I}n: {A}utomorphic {F}orms, automorphic {R}epresentations, and {A}rithmetic}, Proc. Symp. Pure Math., 66 (1999), pp.~239--250.

\bibitem{iwaniec2004analytic}
{\sc H.~Iwaniec and E.~Kowalski}, {\em Analytic number theory}, vol.~53, Amer. Math. Soc., 2004.

\bibitem{soundararajan2000nonvanishing}
{\sc K.~Soundararajan}, {\em Nonvanishing of quadratic {D}irichlet ${L}$-functions at $s = 1/2$}, Ann. Math., 152 (2000), pp.~447--488.

\bibitem{soundararajan2008extreme}
{\sc K.~Soundararajan}, {\em Extreme values of zeta and ${L}$-functions}, Math. Ann., 342 (2008), pp.~467--486.

\bibitem{young2009}
{\sc M.~P. Young}, {\em The first moment of quadratic {D}irichlet ${L}$-functions}, Acta Arith., 138 (2009), pp.~73--99.

\end{thebibliography}
\end{document}